\documentclass{amsart}

\usepackage{dsfont}	  	    
\usepackage{mathrsfs}       

\title{On the Role of Cylindrical Functions in Kantorovich duality}
\author{Martijn Zaal}
\date{\today}
\address[Martijn Zaal]{Institut f\" ur Angewandte Mathematik, University of Bonn \newline \indent Endenicher Allee 60, 53115 Bonn, Germany}
\email{mzaal@iam.uni-bonn.de}
\urladdr{http://www.iam.uni-bonn.de/users/zaal/}
\date{\today}
\subjclass{46G12, 49N15}
\keywords{optimal transport, Kantorovich duality}

\DeclareMathOperator{\Cyl}{\mathrm{Cyl}}
\newcommand{\Naturals}{\mathds{N}}
\newcommand{\Reals}{\mathds{R}}
\newcommand{\set}[1]{\{#1\}}
\newcommand{\Set}[1]{\left\{#1\right\}}
\newcommand*{\ud}[1]{\,\mathrm{d}#1}

\newtheorem{theorem}{Theorem}[section]
\newtheorem{lemma}[theorem]{Lemma}
\newtheorem{proposition}[theorem]{Proposition}
\theoremstyle{remark}
\newtheorem{assumption}[theorem]{Assumption}
\newtheorem*{remark}{Remark}

\numberwithin{equation}{section}

\begin{document}

\begin{abstract}
    We study the dual formulation of the Monge-Kantorovich optimal transportation problem, in particular under what circumstances it is permitted in an infinite dimensional setting to use cylindrical functions, i.e. functions of the form $\varphi\circ P$ where $P$ is a finite-rank operator and $\varphi$ is a smooth, compactly supported function. 
    In the last section, some examples of applications are presented.
\keywords{optimal transport, Kantorovich duality, approximation property}
\subjclass{46G12, 49N15}
\end{abstract}

\maketitle


\section{Introduction}
\label{s:introduction}

Even though optimal transportation problem, first studied by Monge in the paper \emph{M\'emoire sur la th\'eorie des d\'eblais et de remblais} in 1781 \cite{Monge-1781}, is easy to state, it gives rise to rich theory and has proven useful in many applications. 
Only in 1942, Kantorovich provided an answer to the existence question by studying a relaxed version of the problem \cite{Kantorovich-1942}. 
In present time, the Monge-Kantorovich problem is understood to be the search for a probability measure $\gamma$ on the product $X\times Y$ with prescribed marginals $\mu$ and $\nu$ on $X$ and $Y$ respectively, minimizing the cost
\begin{equation}\label{eq:Monge-Kantorovich}
    \int_{X\times Y}c(x, y)\ud{\gamma(x, y)}.
\end{equation}
Since the late 1980's, the Monge-Kantorovich optimal transport problem has received renewed attention and proved to be useful in many fields of mathematics including dynamical systems, differential geometry and probability theory.

For a survey of the Monge-Kantorovich problem and its applications, the reader is referred to, for example, \cite{Bogachev-Kolesnikov-2012}, \cite{Villani-Topics} or \cite{Villani-OldNew}.

Also introduced by Kantorovich is a problem dual to the minimization of \eqref{eq:Monge-Kantorovich}: the maximization of 
\begin{equation}\label{eq:Kantorovich}
    \int_X\phi\ud{\mu}+\int_Y\psi\ud{\nu}
\end{equation}
among all pairs $(\phi, \psi)$ of integrable functions such that $\phi(x)+\psi(y)\le c(x, y)$ everywhere. 
Being closely related to the Monge-Kantorovich optimal transport problem, the dual problem serves as a valuable tool for analyzing problems and questions related to the Monge-Kantorovich optimal transport problem.

The class of functions over which \eqref{eq:Kantorovich} is to be maximized can be reduced considerably (see also \cite[Section 1.1]{Villani-Topics}): the supremum does not change when $\phi$ and $\psi$ are required to be continuous in addition. 
In case the spaces $X$ and $Y$ are Polish, that is, complete separable metric spaces, $\phi$ and $\psi$ may also be assumed to be bounded.
Finally, if $X$ and $Y$ are locally compact, $\phi$ and $\psi$ may be required to decay to zero.
One should note, however, that a maximizer of \eqref{eq:Kantorovich}, if it exists, may not possess this regularity.

We will focus on the case where $X$ and $Y$ are vector spaces. 
The case $X=Y=\Reals^n$ with a cost function of the form $h(x-y)$ seems to be well-studied (see, for instance \cite{Gangbo-McCann-1996}). 
In case $X$ or $Y$ is infinite dimensional, much less seems to be known. 
The Monge-Kantorovich problem when $X=Y$ is an abstract Wiener with a specific cost has been studied by Feyel and \"Ust\"unel \cite{Feyel-Ustunel-2003}, but this reference does not explore the structure of the dual problem. 
To the author's knowledge, no results but the general duality results discussed above are known for infinite dimensional vector spaces.
Here, we will show that, under certain circumstances, the class of functions used for the Kantorovich dual problem can be limited to cylindrical functions, more precisely functions of the form $\varphi\circ P$, where $P$ is a finite rank operator and $\varphi$ is a smooth function with compact support.

The organization of this paper is as follows: in Section \ref{s:preliminaries}, some preliminary results concerning duality in Polish spaces and Euclidean space are recalled.
The main results are stated and proven in Section \ref{s:cylinder}.
Finally, in Section \ref{s:examples}, some examples are presented.


\section{Preliminaries}
\label{s:preliminaries}

\subsection{General duality}

In this section, we will briefly consider the general Kantorovich duality theorem. 
Here and in what follows, we assume that $X$ and $Y$ are Polish spaces, and $c$ is a nonnegative, lower semicontinuous function. 
The value $+\infty$ is permitted for $c$.

\begin{theorem}[Kantorovich duality]\label{th:kantorovich}
    Define the set $\Phi_c$ to be the set of pairs $(\phi, \psi)\in L^1(\mu)\times L^1(\nu)$ such that 
    \begin{equation}\label{eq:constraint}
        \phi(x)+\psi(y)\le c(x, y).
    \end{equation}
    for $\mu$-almost every $x$ and $\nu$-almost every $y$. 
    Then
    \begin{equation}\label{eq:duality}
        \mathscr{C}_c(\mu, \nu)=\sup_{(\phi, \psi)\in\Phi_c}\Set{\int_X\phi\ud{\mu}+\int_Y\psi\ud{\nu}}.
    \end{equation}
    Moreover, in \eqref{eq:duality}, the supremum may be taken over $\Phi_c\cap\left(C^0_b(X)\times C^0_b(Y)\right)$.
\end{theorem}
See, for instance, \cite[Theorem 1.3]{Villani-Topics}.
Actually, going through the proof of this theorem, it follows that the supremum can be limited to those $(\phi, \psi)$ that satisfy \eqref{eq:constraint} \emph{everywhere}.

By some carefully chosen modifications of an admissible pair $(\phi, \psi)$, that is, a pair satisfying \eqref{eq:constraint}, the supremum in \eqref{eq:duality} can be computed over a smaller class of functions.
More precisely, a number of additional bounds may be assumed for $\phi$ and $\psi$.

\begin{lemma}\label{th:kantorovich-bounded}
    Let $\Phi_c^*$ be the set of all pairs $(\phi, \psi)\in C^0_b(X)\times C^0_b(Y)$ such that 
    \begin{gather}
        \inf_X\phi=0, \qquad
        \sup_X\phi\le\sup_{X\times Y}c, \qquad
        \psi\ge\inf_{x\in X, y\in Y}\Set{c(x, y)-\phi(x)}, \qquad
        \sup_Y\psi\ge 0, \label{eq:potential-bounds}\\
        \inf_{x\in X, y\in Y}\Set{c(x, y)-\phi(x)-\psi(y)}=0\label{eq:sharp-constraint}.
    \end{gather}
    Then 
     \begin{equation}
        \mathscr{C}_c(\mu, \nu)=\sup_{(\phi, \psi)\in\Phi_c^*}\left\{\int_X\phi\,\mathrm{d}{\mu}+\int_Y\psi\,\mathrm{d}{\nu}\right\}
    \end{equation}
\end{lemma}
\begin{proof}
    Let $(\phi, \psi)\in C^0_b(X)\times C^0_b(Y)$ satisfying \eqref{eq:constraint} be given, and write
    \begin{equation}
        \psi^c(x):=\inf_Y\left(c(x, .)-\psi\right),\qquad
        \lambda:=\sup_X\left(\phi-\psi^c\right)+\inf_X\psi^c
    \end{equation}
    Since $c$ is bounded from below and $\phi$ and $\psi$ are bounded, $\psi^c$ is bounded from below and $\lambda$ is finite. 
    Let then $\hat\phi:=(\phi-\lambda)^+$. 
    As $\lambda\ge\inf_X\phi$, we have $\inf_X\hat\phi=0$, and clearly $\inf_X(\psi^c-\hat\phi)\le\inf_X\psi^c$. 
    By construction,
    \begin{equation}\label{eq:potential-shift-bound}
        \phi-\lambda\le\psi^c-\inf_X\psi^c.
    \end{equation}
    As the right hand side is nonnegative, it follows that $\hat\phi\le\psi^c-\inf_X\psi^c$, that is $\psi^c-\hat\phi\ge\inf_X\psi^c$. 
    We conclude that $\inf_X(\psi^c-\hat\phi)=\inf_X\psi^c$.
    
    Now define
    \begin{equation}
        \hat\psi:=\max\Set{\psi+\inf_X\psi^c, \inf_{x\in X, y\in Y}\Set{c(x, y)-\hat\phi(x)}}
    \end{equation}
    Again, note that bounds on $c$, $\hat\phi$ and $\psi$ imply that $\hat\psi$ is well-defined. 
    Since $\inf_X(\psi^c-\hat\phi)=\inf_X\psi^c$, we have
    \begin{equation}
        c(x, y)-\hat\phi(x)-\psi(y)-\inf_X\psi^c
            \ge \psi^c(x)-\hat\phi(x)-\inf_X\psi^c
            \ge 0
    \end{equation}
    for any $x\in X$, $y\in Y$, and
    \begin{equation}
        \inf_{x\in X, y\in Y}\Set{c(x, y)-\hat\phi(x)-\inf_{\xi\in X, \eta\in Y}\Set{c(\xi, \eta)-\hat\phi(\xi)}}=0.
    \end{equation}
    It follows that
    \begin{equation}
        \inf_{x\in X, y\in Y}\Set{c(x, y)-\hat\phi(x)-\hat\psi(y)}=0.
    \end{equation}
    
    Given $\varepsilon>0$, there are $x\in X$ such that $\psi^c(x)\le\inf_X\psi^c-\frac{\varepsilon}{2}$. 
    By \eqref{eq:potential-shift-bound}, $\hat\phi(x)\le\frac{\varepsilon}{2}$.
    For any $y\in Y$ such that
    \begin{equation}
        c(x, y)-\psi(y)\le\psi^c(x)+\frac{\varepsilon}{2}\le\inf_X\psi^c+\varepsilon
    \end{equation}
    it now follows that
    \begin{equation}
        \hat\psi(y)\ge\psi(y)+\inf_X\psi^c\ge c(x, y)-\varepsilon\ge-\varepsilon,
    \end{equation}
    that is, $\sup_Y\hat\psi\ge 0$.
    Consequently, 
    \begin{equation}
        \hat\phi
            \le\inf_{y\in Y}\Set{c(., y)-\hat\psi(y)}
            \le\sup_{y\in Y} c(., y)-\sup_Y\hat\psi
            \le\sup_{X\times Y}c.
    \end{equation}
    Finally,
    \begin{equation}
        \int_X\hat\phi\ud{\mu}+\int_Y\hat\psi\ud{\nu}
            \ge\int_X\phi\ud{\mu}-\lambda+\int_Y\psi\ud{\nu}+\inf_X\psi^c
            \ge\int_X\phi\ud{\mu}+\int_Y\psi\ud{\nu}
    \end{equation}
    since
    \begin{equation}
        \inf_X\psi^c-\lambda
            =-\sup_X(\phi-\psi_c)
            =\inf_{x\in X, y\in Y}\Set{c(x, y)-\phi(x)-\psi(y)}
            \ge 0.
    \end{equation}
    We conclude by applying Theorem \ref{th:kantorovich}.
\end{proof}
\begin{remark}
    Combining the bounds in \eqref{eq:potential-bounds} and \eqref{eq:sharp-constraint}, some more bounds can be derived:
    \begin{equation}\label{eq:potential-extended-bounds}
        -\sup_{X\times Y}c\le-\sup_X\phi\le\psi\le\sup_{X\times Y}c,  \qquad
        \psi\le\inf_{x\in X}\Set{c(x, .)-\phi(x)}\le\inf_{x\in X}\Set{c(x, .)}
    \end{equation}
    whenever $(\phi, \psi)\in\Phi_c^*$.
\end{remark}

In what follows, it will prove useful to consider a slightly generalized version of $\Phi_c^*$. 
For $\delta>0$, we will denote by $\Phi_c^\delta$ the set of pairs $(\phi, \psi)\in\Phi_c\cap(C^0_b(X)\times C^0_b(Y))$ such that
\begin{gather}
    -\delta\le\inf_X\phi\le 0, \qquad
    \sup_X\phi\le\sup_{X\times Y}c, \qquad
    \psi\ge-\sup_X\phi-\delta, \qquad
    \sup_Y\psi\ge-\delta. \label{eq:potential-bounds-relaxed} 
\end{gather}
From this and \eqref{eq:constraint}, relaxed versions of \eqref{eq:potential-extended-bounds} can be derived.

\subsection{Euclidean space}
\label{s:euclidean}

In this section, we will prove that the supremum in the Kantorovich dual problem may be restricted to compactly supported smooth functions when $X$ and $Y$ are Euclidean spaces. 
The proof of the result below is adapted from \cite[Section 2]{Natile-Peletier-Savare-2010}.

As will be come clear in the argument below, some coercivity for the cost function is needed.
Clearly, it is not realistic to assume that sublevels of the cost $c$ are compact: even in case $X=Y=\Reals$ and $c(x, y)=|x-y|^p$ for some $p>0$, the cost is equal to zero on the unbounded set $\set{(x, x):x\in \Reals}$. 
However, such a cost does satisfy the following assumption.

\begin{assumption}[Local coercivity]\label{as:coercive-euclidean}
    Given any $R>0$, 
    \[
        0<\delta<\inf_{|x|<R}\sup_{y\in Y}\set{c(x, y)}
    \]
    and $M>0$, there exists an $R'>0$ such that $|x|<R$ and 
    \[
        c(x, y)<\min\Set{\sup_{\eta\in Y}\set{c(x, \eta)}-\delta, M}
    \]
    implies $|y|<R'$.
\end{assumption}

\begin{lemma}\label{th:kantorovich-smooth}
    Suppose $c(x, y):\Reals^m\times \Reals^n\to [0, +\infty)$ is lower semi-continuous and satisfies Assumption \ref{as:coercive-euclidean}. 
    Then
    \begin{equation}
        \mathscr{C}_c(\mu, \nu)
            =\sup_{(\phi, \psi)\in\Phi_c^\delta\cap(C^\infty_c(\Reals^m)\times C^\infty_c(\Reals^n))}\left(\int_{\Reals^m}\phi(x)\ud{\mu}+\int_{\Reals^n}\psi(y)\ud{\nu}\right)
    \end{equation}
    for all $\mu\in\mathscr{P}(X), \nu\in\mathscr{P}(Y)$ and $\delta>0$.
\end{lemma}
\begin{proof}
    Throughout the proof, let $\eta:[0, +\infty)\to[0, 1]$ be a smooth function such that $\eta(z)=1$ for $z\le\frac{1}{2}$ and $\eta(z)=0$ for $z\ge 1$, and denote by $B_r\subset\Reals^k$ the ball of radius $r>0$ centered at $0$.
    For $k\in\Naturals$, $\rho>0$, define $\eta^k_\rho\in C^\infty_c(\Reals^k)$ by
    \begin{equation}
        \eta^k_\rho(x)=\frac{\eta\left(\frac{|x|}{\rho}\right)}{\alpha_k\rho^k},\qquad
        \alpha_k:=\int_{\Reals^k}\eta(|x|)\ud{x}.
    \end{equation}

    Let then $0<\varepsilon<\delta$, and suppose that $(\phi, \psi)\in\Phi_c^*$. 
    There exists $R_x>0$ be such that
    \begin{equation}
        \int_{\Reals^m\setminus B_{R_x-1}}\phi\ud{\mu}<\frac{\varepsilon}{6}.
    \end{equation}
    As $\phi$ is continuous, $\eta^m_\rho\phi\to\phi$ uniformly on bounded sets as $\rho\downarrow 0$. 
    In particular, there exists $\rho_x\in(0, 1)$ such that
    \begin{equation}
        \sup_{B_{R_x}}\left(\eta^m_{\rho_x}*\phi-\phi\right)^+<\frac{\varepsilon}{6}.
    \end{equation}
    Define then 
    \begin{equation}
        \hat\phi:=\eta^m_{\rho_x}*\left(\phi\chi_{B_{R_x-\rho_x}}-\tfrac{\varepsilon}{3}\chi_{B_{R_x+\rho_x}}\right),
    \end{equation}
    where $\chi_E$ is the characteristic function of $E$, defined by $\chi_E(x)\equiv 1$ on $E$ and $\chi_E(x)\equiv 0$ outside $E$.
    By construction $\hat\phi$ is smooth, and supported in $B_{R_x+2\rho_x}$. 
    Moreover, 
    \begin{equation}
        \hat\phi
            \le\eta^m_{\rho_x}\phi-\tfrac{\varepsilon}{3}
            \le\phi-\tfrac{\varepsilon}{6}
    \end{equation}
    in $B_{R_x}$. 
    As $\hat\phi\le 0$ outside $B_{R_x}$, $\hat\phi\le\phi$ everywhere.
    Finally, $\hat\phi\ge-\frac{\varepsilon}{3}$ everywhere.
    
    Since $\sup\psi\ge 0$, 
    \begin{equation}
        \phi(x)\le\sup_{y\in\Reals^n}\Set{c(x, y)}.
    \end{equation}
    By Assumption \ref{as:coercive-euclidean}, there exists $R_y>0$ such that $|x|<R_x$ and $|y|\ge R_y$ implies
    \begin{equation}\label{eq:potential-estimate}
        c(x, y)
            \ge\min\Set{\sup_{\eta\in Y}\set{c(x, \eta)}, \sup_X\phi}-\tfrac{\varepsilon}{6}
            \ge\min\Set{\phi(x), \sup_X\phi}-\tfrac{\varepsilon}{6}
            \ge\phi(x)-\tfrac{\varepsilon}{6}.
    \end{equation}
    Without loss of generality, it may also be assumed that
    \begin{equation}
        \int_{\Reals^n\setminus B_{R_y-1}}\psi\ud{\nu}<\frac{\varepsilon}{4}.
    \end{equation}
    As before, there exists $\rho_y\in(0, 1)$ such that
    \begin{equation}
        \sup_{B_{R_y}}\left(\eta^n_{\rho_y}*\psi-\psi\right)^+<\frac{\varepsilon}{4}.
    \end{equation}
    Again, defining 
    \begin{equation}
        \hat\psi:=\eta^n_{\rho_y}*\left(\psi\chi_{B_{R_y-\rho_y}}-\tfrac{\varepsilon}{4}\chi_{B_{R_y+\rho_y}}\right)
    \end{equation}
    results in $\hat\psi\le\psi$ on $B_{R_y}$, and $\hat\psi\le 0$ outside $B_{R_y}$.
    Moreover, $\hat\psi\ge\psi-\frac{\varepsilon}{4}$ is smooth and compactly supported.
    By construction and \eqref{eq:potential-estimate},
    \begin{equation}\begin{split}
        \hat\phi(x)+\hat\psi(y)
        &   \le\begin{cases}
                \phi(x)+\psi(y)-\tfrac{\varepsilon}{6}, &  \text{if $|x|<R_x$, $|y|<R_y$, } \\
                \phi(x)-\tfrac{\varepsilon}{6}, &          \text{if $|x|<R_x$, $|y|\ge R_y$, } \\
                \psi(y), &          \text{if $|x|\ge R_x$, $|y|<R_y$, } \\
                0, &                \text{otherwise, }
            \end{cases} \\
        &   \le c(x, y),
    \end{split}\end{equation}
    and $(\phi, \psi)\in\Phi_c^\delta$.
    Moreover,
    \begin{equation}\begin{split}
        \int_{\Reals^m}\hat\phi\ud{\mu}+\int_{\Reals^n}\hat\psi\ud{\nu}
        &   \ge\int_{B_{R_x}}\phi\ud{\mu}-\frac{\varepsilon}{3}+\int_{B_{R_y}}\psi\ud{\nu}-\frac{\varepsilon}{4} \\
        &    >\int_{\Reals^m}\phi\ud{\mu}+\int_{\Reals^n}\psi\ud{\nu}-\varepsilon.
    \end{split}\end{equation}
    An application of Lemma \ref{th:kantorovich-bounded} now concludes the proof.
\end{proof}


\section{Fr\'echet space}
\label{s:cylinder}

\subsection{Fr\'echet Spaces and Finite Dimensional Approximations}

In this section, we focus on the case where $X$ and $Y$ are Fr\'echet spaces.
Remember that a locally convex topological vector space is called a Fr\'echet space if its topology is generated by a translation invariant metric that makes the space complete.
In the setting of Fr\'echet spaces, one must be careful with the word `bounded': as many different metrics may generate the topology, boundedness is not well-defined in the usual metric sense.
Instead, a subset $W$ of $X$ is bounded of for every open neighborhood $U$ of $0$, $W\subset tU$ for some $t>0$.

In order to properly define the class of cylindrical functions on a Fr\'echet space $X$, it is necessary to define some kind of coordinate system.
One obvious choice would be to demand that $X$ has a Schauder basis. 
It is possible, however, to consider a more general class of spaces: we will consider spaces that have the so-called approximation property, that is, spaces that admit a family $\set{P_k:X\to X}_{k\in\Naturals}$ of continuous finite rank operators such that $P_k$ converges to the identity uniformly on precompact subsets of $X$.

\begin{remark}
    In fact, it is sufficient to assume that convergence of $P_k$ to the identity is pointwise, since pointwise convergence of linear maps from a Fr\'echet space to any topological vector space implies uniform convergence on compact subsets.
    This fact is easily proven from the Banach-Steinhaus theorem \cite[2.5]{Rudin-Functional-Analysis}.
    It also show that $\set{P_k}_{k\in\Naturals}$ is equicontinuous.
\end{remark}
\begin{remark}
    Note that any Fr\'echet space with the approximation property is also separable and hence a Polish space.
\end{remark}

By definition, the range of a finite rank operator is isomorphic to some Euclidean space, that is, there exists $n\in\Naturals$ and a linear bijection $\iota_P:P[X]\to\Reals^n $ that is continuous and has continuous inverse. 
In what follows, we will not make the distinction between $P$ and $\iota_P\circ P[X]$ if there is not chance of confusion.
We write
\begin{equation}
    \Cyl(X; P):=\Set{\varphi\circ P:\varphi\in C^\infty_c(\Reals^n)},
\end{equation}
and call the elements of $\Cyl(X; P)$ cylindrical functions with respect to $P$. 
A function is called cylindrical if it is in $\Cyl(X; P)$ for some $P$.

In the remainder of this section, we will assume that $X$ and $Y$ are Fr\'echet spaces, that $\set{P_k}_{k\in\Naturals}$ and $\set{Q_k}_{k\in\Naturals}$ are families of continuous finite rank operators converging to the identity map on $X$ and $Y$, respectively.

\subsection{Statement and Proof of the Main Results}

In order to apply the results from Section \ref{s:euclidean}, some requirement similar to Assumption \ref{as:coercive-euclidean} is needed. 
More precisely, Lemma \ref{th:kantorovich-smooth} will be applied to the cost composed with $P_k$ and $Q_k$, which means that this function has to satisfy Assumption \ref{as:coercive-euclidean}. 
One can make the following assumption on $c$ itself to ensure that this is the case.

\begin{assumption}\label{as:coercive-frechet}
    Given any bounded $V\subset X$, $0<\delta<\inf_{x\in B_R(x)}\sup_{y\in Y}\set{c(x, y)}$ and $M>0$, there exists a bounded set $W\subset Y$ such that $x\in V$ and
    \[
        c(x, y)<\min\Set{\sup_{\eta\in Y}\set{c(x, \eta)}-\delta, M}
    \]
    implies $y\in W$.
\end{assumption}
In what follows, we will not assume that this assumption holds, but use a logically weaker requirement. 
We will see later, though, that Assumption \ref{as:coercive-frechet} holds for some interesting examples.

\begin{theorem}\label{th:Kantorovich-lower}
    For $k\in\Naturals$, define $c_k:X\times Y\to [0, +\infty]$ by $c_k(x, y)=c(P_k x, Q_k y)$, and suppose that the restriction of $c$ to $P_k[X]\times Q_k[Y]$ satisfies Assumption \ref{as:coercive-euclidean}. 
    Then
    \begin{equation}\label{eq:Kantorovich-lower}
        \mathscr{C}_c(\mu, \nu)\le\liminf_{k\to\infty}\sup_{(\phi, \psi)\in\Phi_{c_k}^\delta\cap(\Cyl(X; P_k)\times\Cyl(Y; Q_k)}\Set{\int_X\phi\ud{\mu}+\int_Y\psi\ud{\nu}}
    \end{equation}
    for any $\mu\in\mathscr{P}(X)$, $\nu\in\mathscr{P}(Y)$ and any $\delta>0$.
\end{theorem}
\begin{proof}
    Let $\varepsilon>0$, and choose $(\phi, \psi)\in\Phi_c^*$ such that
    \begin{equation}
        \mathscr{C}_c(\mu, \nu)\le\int_X\phi\ud{\mu}+\int_Y\psi\ud{\nu}+\delta.
    \end{equation}
    By assumption, $\phi\circ P_k\to\phi$, and $\psi\circ Q_k\to\psi$.
    As $\phi$ and $\psi$ are bounded, we obtain
    \begin{equation}\label{eq:liminf:almost-maximizer}
        \mathscr{C}_c(\mu, \nu)
            \le\lim_{k\to\infty}\left(\int_X\phi\circ P_k\ud{\mu}+\int_Y\psi\circ Q_k\ud{\nu}\right)+\varepsilon.
    \end{equation}
    Since $(\phi, \psi)$ satisfies \eqref{eq:constraint} for all $x\in X$, $y\in Y$, the restrictions of $\phi$ and $\psi$ to $P_k[X]$ and $Q_k[Y]$, respectively, satisfy \eqref{eq:constraint} as well. 
    Therefore, Lemma \ref{th:kantorovich-smooth} implies
    \begin{multline}
        \int_X\phi\circ P_k\ud{\mu}+\int_Y\psi\circ Q_k\ud{\nu} \\
        \begin{aligned}
        &   =\int_{\Reals^m}\phi\ud{(P_k)_\#\mu}+\int_{\Reals^n}\psi\ud{(Q_k)_\#\nu} \\
        &   \le\sup_{(\hat\phi, \hat\psi)\in\Phi_c^\delta\cap(C^\infty_c(\Reals^m)\times C^\infty_c(\Reals^n))}\Set{\int_{\Reals^m}\hat\phi\ud{(P_k)_\#\mu}+\int_Y\hat\psi\ud{(Q_k)_\#\nu}}.
    \end{aligned}\end{multline}
    Now $\tilde\phi:=\hat\phi\circ P_k\in\Cyl(X; P_k)$ and $\tilde\psi:=\hat\psi\circ Q_k\in\Cyl(Y; Q_k)$ whenever $\hat\phi\in C^\infty_c(\Reals^m)$, $\hat\psi\in C^\infty_c(\Reals^n)$. 
    Moreover, $(\tilde\phi, \tilde\psi)\in\Phi_{c_k}^\delta$ by construction.
    Therefore
    \begin{multline}
        \int_X\phi\circ P_k\ud{\mu}+\int_Y\psi\circ Q_k\ud{\nu} \\
            \le\sup_{(\tilde\phi, \tilde\psi)\in\Phi_{c_k}^\delta\cap(\Cyl(X; P_k)\times\Cyl(Y: Q_k))}\Set{\int_{\Reals^m}\tilde\phi\ud{\mu}+\int_Y\tilde\psi\ud{\nu}}.
    \end{multline}
    Combining this with \eqref{eq:liminf:almost-maximizer} results in
    \begin{equation}
        \mathscr{C}_c(\mu, \nu)
            \le\liminf_{k\to\infty}\sup_{(\tilde\phi, \tilde\psi)\in\Phi_{c_k}^\delta\cap(\Cyl(X; P_k)\times\Cyl(Y: Q_k))}\Set{\int_{\Reals^m}\tilde\phi\ud{\mu}+\int_Y\tilde\psi\ud{\nu}}+\varepsilon.
    \end{equation}
    As $\varepsilon>0$ is arbitrary, we conclude.
\end{proof}

In order to obtain a lower bound, more structure has to be assumed for $c$. 
One possibility is to assume that $c$ is continuous and its growth can be controlled in a certain way.
Under these assumptions, the converse of \eqref{eq:Kantorovich-lower} can be shown by considering the Monge-Kantorovich optimal transport problem, rather than the dual problem.

\begin{proposition}\label{th:Kantorovich-upper}
    Suppose that, additionally, $c$ is upper semicontinuous (and hence continuous), and there exists nondecreasing functions $f, g:[0, +\infty)\to[0, +\infty)$ such that
    \begin{equation}
        c(\lambda x, y)\le f(\lambda)c(x, y), \qquad
        c(x, \lambda y)\le g(\lambda)c(x, y),
    \end{equation}
    for any $\lambda>0$, $x\in X$ and $y\in Y$.
    Then 
    \begin{equation}\label{eq:Kantorovich-upper}
        \mathscr{C}_c(\mu, \nu)
            \mathscr{C}_c(\mu, \nu)\ge\limsup_{k\to\infty}\sup_{(\phi, \psi)\in\Phi_{c_k}^\delta\cap(\Cyl(X; P_k)\times\Cyl(Y; Q_k)}\Set{\int_X\phi\ud{\mu}+\int_Y\psi\ud{\nu}}
    \end{equation}
\end{proposition}
\begin{proof}
    In case $\mathscr{C}_c(\mu, \nu)=+\infty$ there is nothing to prove, so assume that $\gamma\in\Gamma(\mu, \nu)$ is such that
    \begin{equation}
        \int_{X\times Y}c(x, y)\ud{\gamma(x, y)}<+\infty.
    \end{equation}
    By assumption, $c(P_k x, Q_k y)\le f(\lambda)g(\lambda)c(x, y)$, which means that
    \begin{equation}\begin{split}
        \int_{X\times Y}c(x, y)\ud{\gamma(x, y)}
        &   \ge\int_{X\times Y}\limsup_{k\to\infty}c(P_k x, Q_k y)\ud{\gamma(x, y)} \\
        &   \ge\limsup_{k\to\infty}\int_{X\times Y}c(P_k x, Q_k y)\ud{\gamma(x, y)} \\
        &   \ge\limsup_{k\to\infty}\mathscr{C}_c\left((P_k)_\#\mu, (Q_k)_\#\nu\right)
    \end{split}\end{equation}
    As $P_k[X]$ and $Q_k[Y]$ isometrically isomorphic to some finite-dimensional Euclidean space, the right hand side is equal to the right hand side of \eqref{eq:Kantorovich-upper} by Theorem \ref{th:kantorovich-smooth}.
    Taking the infimum over $\gamma\in\Gamma(\mu, \nu)$ now implies the result.
\end{proof}

Combining the previous two results implies that the optimal transport cost can be written as the limit of suprema over classes of cylindrical functions. 
However, if the operators $P_k$ and $Q_k$ can be chosen such that $c_k$ is in fact at most $c$, it is no longer needed to take a limit of suprema.

\begin{theorem}\label{th:Kantorovich-cylindrical}
    In case $c(P_k x, Q_k y)\le c(x, y)$ for all $k\in\Naturals$, $x\in X$ and $y\in Y$,
    \begin{equation}\label{eq:Kantorovich-cylindrical}
        \mathscr{C}_c(\mu, \nu)
            =\sup_{k\in\Naturals, (\phi, \psi)\in\Phi_c^\delta\cap(\Cyl(X; P_k)\times\Cyl(Y; Q_k)}\Set{\int_X\phi\ud{\mu}+\int_Y\psi\ud{\nu}}
    \end{equation}
    whenever $\mu\in\mathscr{P}(X)$, $\nu\in\mathscr{P}(Y)$ and $\delta>0$.
\end{theorem}
\begin{proof}
    By assumption, $\Phi_{c_k}^\delta\subset\Phi_c^\delta$ for all $k\in\Naturals$, $\delta>0$. 
    Therefore,
    \begin{equation}\begin{split}
        \mathscr{C}_c(\mu, \nu)
        &   =\sup_{(\phi, \psi)\in\Phi_c^\delta}\Set{\int_X\phi\ud{\mu}+\int_Y\psi\ud{\nu}} \\
        &   \ge\sup_{k\in\Naturals, (\phi, \psi)\in\Phi_c^\delta\cap(\Cyl(X; P_k)\times\Cyl(Y; Q_k)}\Set{\int_X\phi\ud{\mu}+\int_Y\psi\ud{\nu}} \\
        &   \ge\sup_{k\in\Naturals, (\phi, \psi)\in\Phi_{c_k}^\delta\cap(\Cyl(X; P_k)\times\Cyl(Y; Q_k)}\Set{\int_X\phi\ud{\mu}+\int_Y\psi\ud{\nu}}.
    \end{split}\end{equation}
    By Theorem \ref{th:Kantorovich-lower}, this implies the result.
\end{proof}
\begin{remark}
    Note that the proof of Proposition \ref{th:Kantorovich-upper} also works under the assumptions of this theorem.
    Therefore, \eqref{eq:Kantorovich-upper} also holds in this case.
\end{remark}


\section{Some examples}
\label{s:examples}

\subsection{Separable Hilbert Space with Translation Invariant Cost}

Let us consider the case $X=Y=H$ some separable Hilbert space with inner product $\langle. ,. \rangle$ and norm $|.|$, and $c(x, y)=h(|x-y|)$ for some lower semicontinuous function $h:[0, +\infty)\to[0, +\infty]$ such that $h(z)\to\sup_{[0, +\infty)}h$ as $z\to+\infty$.

It is not very difficult to check that Assumption \ref{as:coercive-frechet} holds in this case.
Given $M<\sup_{[0, +\infty)}h$, the cost $c(x, y)$ can only be less than $M$ if $x-y$ is restricted to some bounded set.
If $x$ is already restricted to a bounded set, this means that $y$ must also be restricted to a bounded set.
Therefore, Theorem \ref{th:Kantorovich-lower} can be applied in this case.

The maps $P_k$ and $Q_k$ can be constructed from a complete orthonormal sequence $\set{e_j}_{j\in\Naturals}$:
\begin{equation}
    P_k x=Q_k x=\sum_{j=1}^k\langle x, e_j\rangle e_j
\end{equation}

If, additionally, $h$ is continuous, Proposition \ref{th:Kantorovich-upper} can be applied as well.
Therefore, we conclude that
\begin{equation}\label{eq:Kantorovich-limit}
    \mathscr{C}_c(\mu, \nu)=\lim_{k\to\infty}\sup_{(\phi, \psi)\in\Phi_{c_k}^\delta\cap(\Cyl(X; P_k)\times\Cyl(Y; Q_k)}\Set{\int_X\phi\ud{\mu}+\int_Y\psi\ud{\nu}}.
\end{equation}

If we assume $h$ to be nondecreasing instead, the example fits the requirements of Theorem \ref{th:Kantorovich-cylindrical}, and we obtain \eqref{eq:Kantorovich-cylindrical}.

It is easy to see that the similar results hold if
\begin{equation}
    c(x, y)=h\left(\sum_{j=1}^\infty h_j(\langle x, e_j\rangle)\right)
\end{equation}
with $h_j\ge 0$ lower semicontinuous with $h_j(z)\to+\infty$ as $z\to+\infty$ for every $j\in\Naturals$, and $h$ continuous and increasing.
Although $c$ may not satisfy Assumption \ref{as:coercive-frechet}, it is obvious that
\begin{equation}
    c_k(x, y)=h\left(\sum_{j=1}^k h_j(\langle x, e_j\rangle)\right)
\end{equation}
satisfies \ref{as:coercive-euclidean} for every $k$.
Depending on the continuity, growth and summability of the $h_j$, Proposition \ref{th:Kantorovich-upper} may or may not be applicable in this case.
If, however, all $h_j$ are increasing, Theorem \ref{th:Kantorovich-cylindrical} can be applied.

\subsection{Banach Spaces With a Schauder Base}

If $X=Y$ is some Banach space that has a Schauder base, there is a canonical map $P_k$. 
However, since it is not a priori true that $\|P_k\|\le 1$, Theorem \ref{th:Kantorovich-cylindrical} cannot be applied a priori, even if $c(x, y)=\|x-y\|^p$ for some $p>0$.

If $c$ has a structure similar to the structure in the previous section, it is possible to apply Theorem \ref{th:Kantorovich-lower} and Proposition \ref{th:Kantorovich-upper} to obtain \eqref{eq:Kantorovich-limit}.

\subsection{Wiener Space}

An important example of an infinite dimensional vector space where optimal transportation may be of interest is Wiener space (see also \cite{Feyel-Ustunel-2003}.
If $W$ is a separable Fr\'echet space, and $\gamma$ is a Gaussian measure with support $W$, then the Cameron-Martin space $H$ of $\gamma$ is dense in $W$ and separable (see, for example, \cite{Bogachev-Gaussian}).
Therefore, there exists a complete orthonormal sequence $\set{h_k}_{k\in\Naturals}$ in $H$ such that the functionals $(h_k, .)_H$ extend by continuity to $W$. 
For such a sequence, the orthogonal projections $P_k$ onto the span $F_k$ of $\set{h_1, \ldots, h_k}$ defined by
\begin{equation}
    P_k h=\sum_{j=1}^k (h_k, h)_H h_k
\end{equation}
also extend by continuity to $W$.
Moreover, as the sequence $\set{h_k}$ is complete and $H$ is dense in $W$, the union of $F_k$ is dense in $H$. 
However, it is not in general true that $P_k x\to x$ in all of $W$ as $\set{P_k}$ is not necessarily equicontinuous as a collection of maps from $W$ to $W$.

In case $P_k$ does converge pointwise to the identity, the results from the previous section can be applied. 
Particularly interesting is the cost function that $|x-y|_H$ whenever $x-y\in H$ and $+\infty$ otherwise. 
It is not difficult to show that this cost function is lower semicontinuous and that $c(P_k x, P_k y)\le c(x, y)$ for all $x, y\in W$. 
Therefore, Theorem \ref{th:Kantorovich-cylindrical} can be applied in this case.

\bibliographystyle{amsplain}
\bibliography{refs}

\end{document}